\documentclass[a4paper,11pt]{article}
\usepackage{amsmath, amsfonts, amscd, amssymb, amsthm, enumerate, xypic}

\def\LC{\text{LC}}

\def\card{\#\,}

\newcommand{\Vect}{\operatorname{span}}

\newcommand{\supp}{\operatorname{supp}}

\renewcommand{\setminus}{\smallsetminus}

\def\R{\mathbb{R}}

\def\calD{\mathcal{D}}

\def\calN{\mathcal{N}}
\def\calO{\mathcal{O}}
\def\calP{\mathcal{P}}

\def\lcro{\mathopen{[\![}}
\def\rcro{\mathclose{]\!]}}

\theoremstyle{definition}
\newtheorem{Def}{Definition}
\newtheorem{Not}[Def]{Notation}

\theoremstyle{plain}
\newtheorem{theo}{Theorem}
\newtheorem{prop}[theo]{Proposition}

\newtheorem{lemma}[theo]{Lemma}

\theoremstyle{plain}

\theoremstyle{remark}

\title{On the finite topology of a vector space and the domination problem for families of norms}
\author{Cl\'ement de Seguins Pazzis
\footnote{Universit\'e de Versailles Saint-Quentin-en-Yvelines, Laboratoire de Math\'ematiques
de Versailles, 45 avenue des \'Etats-Unis, 78035 Versailles cedex, France}
\footnote{e-mail address: dsp.prof@gmail.com}}

\begin{document}

\thispagestyle{plain}

\maketitle

\begin{abstract}
Let $V$ be a real or complex vector space. 
The finite topology of $V$ consists of all the subsets $U$ 
for which the intersection $U \cap F$ is closed in $F$ for every finite-dimensional linear subspace of $V$. 
It is known that if $V$ has countable dimension, then this topology coincides with the greatest vector space
topology and with the greatest locally convex vector space topology.
Here, we note that in this case the finite topology is simply the union of all the normed space topologies.

In connection to this problem, we characterize the vector spaces on which every family of norms with given cardinality is dominated.
\end{abstract}

\vskip 2mm
\noindent
\emph{AMS Classification:} 46A03, 03E17

\vskip 2mm
\noindent
\emph{Keywords:} normed vector space, topological vector space, finite topology, comparison of norms, ordinals.

\section{Introduction}

In this note, all the vector spaces we consider are over the field of real numbers or over the one of complex numbers, and all 
topological vector spaces are Hausdorff by definition. 
Throughout, $\omega$ denotes the first infinite ordinal, which we identify with the set of all \emph{non-negative} integers, and $\omega_1$ denotes the first uncountable ordinal. The cardinality of a set $I$ is denoted by $\card I$.

Let $V$ be a vector space. A subset $U$ of $V$ is called \textbf{finite open} when $U \cap F$ is open in $F$ for every finite-dimensional
linear subspace $F$ of $V$. The set of all finite open subsets on $V$ is a topology, which we denote by $\tau_f(V)$
and call the \textbf{finite topology} on $V$: it was first introduced by Hille and Phillips \cite{HillePhillips}. 
When $\dim V$ is uncountable, Harremo\"es (see page 547 in \cite{Bisgaard}) has shown that $\tau_f(V)$ is not locally convex, and 
Bisgaard \cite{Bisgaard} improved the result by proving that $\tau_f(V)$ is not a vector space topology. 
However, when $V$ has countable dimension, Kakutani and Klee \cite{Kakutani} proved that $\tau_f(V)$ coincides with the following classical vector space topologies:
\begin{itemize}
\item the greatest vector space topology on $V$, denoted by $\tau_{\max}(V)$;
\item the greatest locally convex topology on $V$, denoted by $\tau^{\LC}_{\max}(V)$.
\end{itemize}
Moreover, it is known that $\tau^{\LC}_{\max}(V)$ differs from $\tau_{\max}(V)$ when $V$ has uncountable dimension
\cite{Jarchow,Wilanski,Zelasko}. Thus, in general $\tau^{\LC}_{\max}(V) \subset \tau_{\max}(V) \subset \tau_f(V)$, 
all the inclusions are strict if $V$ has countable dimension, and they are all equalities otherwise. 

Here, we shall be concerned with another set, which seems to have been neglected. 
Denote by $\calN(V)$ the set of all norms on $V$; for such a norm $N$, denote by $\calO_N$ the associated topology on $V$, and set 
$$\tau_{\calN}(V)=\underset{N \in \calN(V)}{\bigcup} \calO_{N.}$$
Hence, an element of $\tau_{\calN}$ is a subset of $V$ that is open for at least one norm on $V$: we will call such subsets
\textbf{norm-openable}. Obviously,  
$$\tau_{\calN}(V) \subset \tau^{\LC}_{\max}(V) \subset \tau_{\max}(V) \subset \tau_f(V)$$
and a natural question is whether the equality $\tau_{\calN}(V)=\tau^{\LC}_{\max}(V)$ holds. 

Our main result, the proof of which constitutes the first part of the present article, completely answers that question:

\begin{theo}\label{maintheo}
If $V$ has countable dimension, then $\tau_{\calN}(V)=\tau_f(V)$. \\
Otherwise, $\tau_{\calN}(V)$ is not a topology and hence $\tau_{\calN}(V) \neq \tau^{\LC}_{\max}(V)$. 
\end{theo}

In particular, if $V$ has countable dimension then our theorem directly proves the equalities
$$\tau_{\calN}(V) = \tau^{\LC}_{\max}(V) = \tau_{\max}(V) = \tau_f(V)$$
and so the elements of $\tau^{\LC}_{\max}(V)$ are the norm-openable subsets of $V$.
Strangely, we could not find a mention of the latter result in the classical literature, and bridging that 
gap was our first motivation for writing the present article. 

Theorem \ref{maintheo} is deeply connected to the problem of dominating families of norms on $V$.
Recall that a norm $N'$ on $V$ dominates another norm $N$ on $V$ if and only if $N \leq \alpha N'$ for some positive real constant $\alpha$.
We extend this definition to whole families of norms as follows:

\begin{Def}
Let $(N_i)_{i \in I}$ be a family of norms on the vector space $V$.
We say that $(N_i)_{i \in I}$ is \textbf{dominated} if there exists a norm $N'$ on $V$ which dominates $N_i$ for all $i \in I$.
\end{Def}

The proof of Theorem \ref{maintheo} will use the following key observation:

\begin{prop}\label{sequencenormsprop}
Let $V$ be a vector space with dimension $\kappa$. 
If $\kappa$ is countable, then every (countable) sequence of norms on $V$ is dominated. 
Otherwise, there exists an undominated family of norms of $V$ that is indexed over $\kappa$. 
\end{prop}

Proposition \ref{sequencenormsprop} will be proved in Section \ref{normsSECTION}
as a special case of a more general result on the domination of families of norms (see Theorem \ref{normstheo}). 
In the next section, we take Proposition \ref{sequencenormsprop} for granted and we show how it leads to Theorem~\ref{maintheo}. 

Before we proceed, we would like to point to a related open problem. 
Given a point $x \in V$, say that a subset $A$ is a finite neighborhood of $V$ whenever, for every linear subspace $F$
of $V$, $A \cap F$ is a neighborhood of $x$ in $F$. In that situation, is it true that 
$A$ is a neighborhood of $x$ with respect to the finite topology? And, if so, is $A$ a neighborhood of $x$ for some normed topology
$\calO_N$?

\section{Openable subsets vs finite open subsets}\label{maintheosection}

Here, we take Proposition \ref{sequencenormsprop} for granted and we use it to prove Theorem \ref{maintheo}.

\subsection{The case of uncountable-dimensional vector spaces}

Here, we let $V$ be a vector space with uncountable dimension $\kappa$.
By Proposition \ref{sequencenormsprop}, there is an undominated family $(N_k)_{k \in \kappa}$ of norms on $V$.
We shall use those norms to construct a family of norm-openable subsets of $V$ whose union is not norm-openable.

We choose a basis $(e_k)_{k \in \kappa}$ of $V$ and we consider the corresponding supremum norm
$$N' : x \mapsto \sup_{k \in \kappa} |x_k|,$$
where $(x_k)_{k \in \kappa}$ denotes the family of coordinates of the vector $x$ in the basis $(e_k)_{k \in \kappa.}$
Replacing $N_k$ with $N'+N_k$ if necessary, we can assume that $N' \leq N_k$ for all $k \in \kappa$.

Given $k \in \kappa$, we set
$$B_k:=\Bigl\{x\in V : \; N_k(x-e_k) < \frac{1}{3}\Bigr\},$$
the open ball with center $e_k$ and radius $\frac{1}{3}$ with respect to $N_k$. This is an open subset with respect to $N_k$,
whence it is norm-openable. The sets $B_k$ are pairwise disjoint: indeed, for all distinct $k,l$ in $\kappa$, the existence of some $x \in B_k \cap B_l$
would yield
$$N'(e_k-e_l) \leq N'(e_k-x)+N'(x-e_l) \leq N_k(e_k-x)+N_l(x-e_l) \leq \frac{2}{3},$$
which is false.

We claim that the subset
$$O:=\underset{k \in \kappa}{\bigcup} B_k$$
is not norm-openable. Assume on the contrary that there exists a norm $N$ on $V$
for which $O$ is open. We shall prove that $N$ dominates each norm $N_k$.

Let $k \in \kappa$.
There is an open ball $B$ of $(V,N)$, included in $O$ and with center $e_k$.
Let $\calD$ be an arbitrary line of $V$ through $e_k$:
the sets $\calD \cap B_l$, with $l \in \kappa$, are pairwise disjoint open line segments whose union includes
the connected set $\calD \cap B$, whence $\calD \cap B$ is included in one of them. Noting that $e_k \in \calD \cap B$, it follows that
$\calD \cap B \subset \calD \cap B_k$.
Varying $\calD$ yields $B \subset B_k$, and it classically follows that $N$ dominates $N_k$.
This contradicts the assumption that $(N_k)_{k \in \kappa}$ be dominated by no norm, and we conclude that $O$ is not openable. Hence, 
the second statement of Theorem \ref{maintheo} is established. 

\subsection{The case of countable-dimensional vector spaces}

Throughout the present section, $V$ denotes a vector space with countable dimension.
Our aim is to prove the first statement of Theorem \ref{maintheo}. The result is obvious if $V$ is finite-dimensional, so we assume
from now on that $V$ is infinite-dimensional.
Here is our key lemma:

\begin{lemma}\label{L1}
Let $O$ be a finite open subset of $V$, $F$ be a finite-dimensional linear subspace of
$V$ equipped with a norm $N$. Assume that $O$ includes the closed unit ball $B$
of $(F,N)$. Let $x \in V \setminus F$ and set $F':=F \oplus \R x$.
Then, $N$ extends to a norm $N'$ on $F'$ whose closed unit ball is included in~$O$.
\end{lemma}

\begin{proof}
For $(y,\lambda)\in F \times \R$, set $\|y+\lambda x\|:=\max(N(y),|\lambda|)$, thereby defining a norm
$\|-\|$ on $F'$.
Since $B$ is compact and disjoint from the closed subspace $F' \setminus O$ of $F'$, there exists a real number $\epsilon>0$ such that $\|a-b\| > \epsilon$ for all $a \in B$ and all
$b \in F' \setminus O$. Let then $(y,\lambda)\in B \times (-\epsilon,\epsilon)$.
Since $\|y-(y+\lambda x)\| \leq \epsilon$ and $y \in B$, we deduce that $y+\lambda x \in O$.
Now, for  $(y,\lambda)\in F \times \R$, we set $N'(y+\lambda x):=\max(N(y),\epsilon^{-1}|\lambda|)$, thereby defining a norm $N'$ on $F'$ that extends $N$ and whose closed unit ball is included in~$O$.
\end{proof}

\begin{lemma}\label{L2}
Let $O$ be a finite open subset of $V$, $F$ be a finite-dimensional linear subspace of
$V$ equipped with a norm $N$. Assume that $O$ includes the closed unit ball $B$
of $(F,N)$. Then, $N$ extends to a norm $N'$ on $V$ whose closed unit ball is included in~$O$.
\end{lemma}

\begin{proof}
Let $(e_n)_{n \in \omega}$ be a basis of a complementary subspace of $F$ in $V$ and, for all
$n \in \omega$, set $F_n:=F \oplus \Vect(e_k)_{0 \leq k<n}$ (in particular $F_0=F$).
The preceding lemma yields, by induction, a sequence $(N_n)_{n \in \omega}$ such that:
\begin{itemize}
\item $N_0=N$;
\item for all $n \in \omega$, $N_n$ is a norm on $F_n$ whose closed unit ball is included in $O$;
\item $N_{n+1}$ extends $N_n$ for all $n \in \omega$.
\end{itemize}
It is then obvious that the functions $N_n$ extend (uniquely) to a function $N'$ on $V$ that extends $N$, is a norm,
and for which the closed unit ball is the union of the closed unit balls of the $N_n$'s. Hence, that closed unit ball
is included in~$O$.
\end{proof}

We are now ready to prove the first statement in Theorem \ref{maintheo}. Let $O$ be a finite open subset of $V$.
Choose a basis $(e_n)_{n \in \omega}$ of $V$.
Let $n \in \omega$ and set $F_n:=\Vect(e_k)_{0 \leq k \leq n}$. Choose an arbitrary norm $N_0$ on $V$.
Since $O \cap F_n$ is an open subset of $F_n$, there exists, for each $x \in O \cap F_n$, a radius $r_x>0$
such that $O$ includes the closed ball $B(x,r_x)$ of $(F_n,N_0)$, and we define $O_{n,x}$ as the corresponding open ball.
Since $O \cap F_n$ is separable (being a subspace of a finite-dimensional vector space) there exists a countable subset $A_n$ of $O \cap F_n$ such that
$\underset{x \in A_n}{\bigcup} O_{n,x}=O \cap F_n$.
Let $x \in A_n$. Obviously, the translated set $O-x$ is still finite open, and it includes the closed unit ball of $F_n$ for $r_x^{-1} N_0$.
Hence, by Lemma \ref{L2} there exists a norm $N_x$ on $V$ that extends $r_x^{-1} (N_0)_{|F_n}$ and whose closed unit ball is included in $O-x$.
It follows that $O_{n,x}$ is included in an open ball for $N_x$ that is included in $O$.

The set of norms $\{N_x \mid n \in \omega, x \in A_n\}$ is countable: Hence Proposition \ref{sequencenormsprop} yields a norm
$N$ on $V$ that dominates all its elements. We shall conclude by proving that
$O$ is open with respect to $N$. Let $n \in \omega$ and $x \in A_n$: there is an open ball for $N_x$ around $x$
that is included in $O$ and includes $O_{n,x}$, and such a ball is also open for $N$ because $N$ dominates $N_x$. It follows that
$O$ is a neighborhood of every point of $O_{n,x}$ with respect to $N$. Varying $n$ and $x$ and remembering that
$O=\underset{n \in \omega,x \in A_n}{\bigcup} O_{n,x}$,
we conclude that $O$ is open in $(V,N)$. This completes the proof of the first statement in Theorem \ref{maintheo}.

\section{Dominated families of norms}\label{normsSECTION}

Here, we prove Proposition \ref{sequencenormsprop} by giving a complete answer to the following question:
given a vector space $V$ and a set $I$, under what conditions is it true that every family $(N_i)_{i \in I}$
of norms on $V$ (indexed over $I$) is dominated?
The answer involves the so-called \textbf{bounding number}: remember the relation $\leq^*$
on ${}^\omega \omega$ defined as follows:
$$f \leq^* g \; \Leftrightarrow \; \exists n \in \omega : \; \forall k \geq n, \; f(k) \leq g(k).$$
A subset $X$ of ${}^\omega \omega$ is called \textbf{bounded} if there exists $g \in {}^\omega \omega$ such that
$\forall f \in X, \; f \leq^* g$, and \textbf{unbounded} otherwise.
The \textbf{bounding number} $\mathfrak{b}$ is the least cardinality for an unbounded subset of ${}^\omega \omega$.
It is easily seen that $\aleph_0<\mathfrak{b}$. Of course, under the continuum hypothesis we would have
$\mathfrak{b}=\mathfrak{c}$ (the cardinality of $\R$), yet it is known that $\mathfrak{b}<\mathfrak{c}$ is consistent with ZFC
(see Section 2 of \cite{Blass}).
It turns out that $\mathfrak{b}$ is connected to many phenomena in general topology: see the surveys \cite{vanDouwen} and \cite{Vaughan}, and
a recent example of a result of this kind \cite{CWcomplexes}.
What follows is our result on dominated families of norms, and it is clear that it implies Proposition \ref{sequencenormsprop}:

\begin{theo}\label{normstheo}
Let $V$ be a vector space, and $I$ be a set.
For every $I$-family of norms on $V$ to be dominated, it is necessary and sufficient that one of the following
conditions holds:
\begin{enumerate}[(i)]
\item $V$ is finite-dimensional;
\item $I$ is finite;
\item $V$ has countable dimension and $\card I<\mathfrak{b}$;
\item $I$ is countable and $\dim V<\mathfrak{b}$.
\end{enumerate}
\end{theo}

The cornerstone of the proof is a connection between the problem of dominating families of norms
and the problem of dominating functions of two variables by the product of two functions of one variable (see \cite{Hathaway}).

\subsection{Dominated families of functions}

\begin{Def}\label{dominFunction}
Let $X$ be a set.
Let $(f_i)_{i \in I}$ be a family of functions from $X$ to $\R$.
We say that $(f_i)_{i \in I}$ is \textbf{dominated} by some function $g : X \rightarrow \R$ when, for all
$i \in I$, there exists a positive real number $c_i$ such that
$$\forall x \in X, \; |f(x)| \leq c_i\,|g(x)|.$$
\end{Def}

\begin{Def}
Let $X$ and $Y$ be sets. A function $f : X \times Y \rightarrow \R$ is called \textbf{separably dominated}
when there are functions $g : X \rightarrow \R$ and $h : Y \rightarrow \R$ such that
$$\forall (x,y)\in X \times Y, \; |f(x,y)| \leq |g(x)|\,|h(y)|.$$
\end{Def}

The following result is then obvious:

\begin{lemma}\label{sepdomlemma}
Let $X$ be a set.
Let $(f_i)_{i \in I}$ be a family of functions from $X$ to $\R$.

Then, $(f_i)_{i \in I}$ is dominated if and only if the function $(i,x)\in I \times X \mapsto f_i(x)$
is separably dominated.
\end{lemma}

The following result is also easy:

\begin{lemma}\label{equivalenceDomination}
Let $X$ and $Y$ be sets and $f : X \times Y \rightarrow \R$ be a function.
Then, $f$ is separably dominated if and only if there exist functions $G : X \rightarrow \omega$ and $H : Y \rightarrow \omega$
such that $\forall (x,y)\in X \times Y, \; |f(x,y)| \leq \max(G(x),H(y))$.
\end{lemma}

\begin{proof}[Proof of Lemma \ref{equivalenceDomination}]
If there are functions $g : X \rightarrow \R$ and $h : Y \rightarrow \R$ such that
$$\forall (x,y)\in X \times Y, \; |f(x,y)| \leq |g(x)|\,|h(y)|,$$
then
$$\forall (x,y)\in X \times Y, \; |f(x,y)| \leq \max\bigl(\bigl\lceil g(x)^2\bigr\rceil, \bigl\lceil h(y)^2\bigr\rceil\bigr).$$
Conversely, if there are functions $G : X \rightarrow \omega$ and $H : Y \rightarrow \omega$ such that
$$\forall (x,y)\in X \times Y, \; |f(x,y)| \leq \max\bigl(G(x),H(y)\bigr),$$
then
$$\forall (x,y)\in X \times Y, \; |f(x,y)| \leq \,|G(x)+1|\,|H(y)+1|.$$
\end{proof}

Finally, if a function $f : X \times Y \rightarrow \R$ is not separably dominated, then
neither is the integer-valued function $(x,y) \mapsto \lceil f(x,y)\rceil$.

\begin{Not}
Let $X$ and $Y$ be sets. We denote by $\calP(X,Y)$ the assertion that every function
from $X \times Y$ to $\omega$ is dominated, which is equivalent to the assertion that every function from
$X \times Y$ to $\R$ is dominated.
\end{Not}

Obviously, $\calP(X,Y)$ is equivalent to $\calP(Y,X)$. Moreover,
given sets $X'$ and $Y'$ with $\card X' \leq \card X$ and $\card Y' \leq \card Y$,
the implication $\calP(X,Y) \Rightarrow \calP(X',Y')$ holds. In particular, the validity of $\calP(X,Y)$
depends only on the respective cardinalities of $X$ and $Y$.

The following result was recently proved by Hathaway \cite{Hathaway}:

\begin{theo}[Hathaway]\label{HathawayThm}
Let $X$ and $Y$ be sets. Then, $\calP(X,Y)$ holds if and only if one of the following conditions holds:
\begin{enumerate}[(i)]
\item One of $X$ and $Y$ is finite;
\item One of $X$ and $Y$ is countable, and the other one has cardinality less than~$\mathfrak{b}$.
\end{enumerate}
\end{theo}

In particular, $\calP(\omega,\omega)$ holds but $\calP(\omega_1,\omega_1)$ fails.
Since those two cases are the only ones relevant to our study on norm-openable sets and since the proofs are short, we include them below:

\begin{proof}[Proof of $\calP(\omega,\omega)$]
Letting $f : \omega \times \omega \rightarrow \omega$ and setting
$g : n \mapsto \max(f(k,l))_{0 \leq k,l \leq n}$, we note that $g$ is non-decreasing and hence
$$\forall (n,p)\in \omega^2, \; f(n,p) \leq g\bigl(\max(n,p)\bigr) = \max\bigl(g(n),g(p)\bigr).$$
\end{proof}

\begin{proof}[Proof of $\neg \calP(\omega_1,\omega_1)$]
For all $\alpha \in \omega_1$, we choose an injection
$$f_\alpha : \{\beta \in \omega_1 : \; \beta \leq \alpha\} \hookrightarrow \omega.$$
We set
$$F : (\alpha,\beta)\in (\omega_1)^2 \mapsto \begin{cases}
f_\alpha(\beta) & \text{if $\beta \leq \alpha$} \\
0 & \text{otherwise.}
\end{cases}$$
We prove that $F$ is not separably dominated. Assume on the contrary that there are functions $G$ and $H$ in ${}^{\omega_1}\omega$ such that
$$\forall (\alpha,\beta) \in (\omega_1)^2, \; F(\alpha,\beta) \leq \max(G(\alpha),H(\beta)).$$
Since $\omega_1$ is uncountable, there are integers $n$ and $p$ such that $G^{-1}(\lcro 0,n\rcro)$ and $H^{-1}(\lcro 0,p\rcro)$ are uncountable,
and hence unbounded in $\omega_1$. Then, we take an infinite countable subset $B$ of $H^{-1}(\lcro 0,p\rcro)$,
and we successively find $\alpha \in \omega_1$ such that $\forall x \in B, \; x\leq \alpha$, and then
$\alpha' \in G^{-1}(\lcro 0,n\rcro)$ such that $\alpha \leq \alpha'$.
Then we would have $f_{\alpha'}(\beta) \leq \max(n,p)$ for all $\beta \in B$, contradicting the injectivity of~$f_{\alpha'}$.
\end{proof}

\subsection{Dominated families of norms}

\begin{lemma}\label{dominFunctiontoNorm}
Let $(N_i)_{i \in I}$ be a family of norms on the vector space $V$.
Then, $(N_i)_{i \in I}$ is dominated as a family of norms if and only if it is dominated as a family of functions.
\end{lemma}

\begin{proof}
The direct implication is obvious. For the converse, assume that $(N_i)_{i \in I}$ is dominated in the meaning of Definition \ref{dominFunction}.
This yields functions $g : I \rightarrow [0,+\infty)$ and $f : V \rightarrow [0,+\infty)$ such that
$$\forall (i,x)\in I \times V, \; N_i(x) \leq g(i) f(x).$$
Replacing $g$ by $g+1$ if necessary, we can assume that $g$ takes positive values only.
Then, for all $x \in V$, the set $\{N_i(x)/g(i) \mid i \in I\}$ is bounded, and it is then a standard observation that
$$N : x \mapsto \underset{i \in I}{\sup} \,\frac{N_i(x)}{g(i)}$$
is a norm on $V$. Obviously $\forall i \in I, \; N_i \leq g(i) N$, whence $N$ dominates all norms~$N_i$.
\end{proof}

We turn to our main result. It is clear that combining it with Theorem \ref{HathawayThm} yields Theorem \ref{normstheo}:

\begin{prop}\label{PtoNorms}
Let $\kappa$ be a cardinal and $I$ be a set. Let $V$ be a vector space with dimension $\kappa$.
Then $\calP(\kappa,I)$ holds if and only if every $I$-family of norms on $V$
is dominated.
\end{prop}

\begin{proof}
In a finite-dimensional vector space all norms are equivalent. Hence, both conditions hold whenever $\kappa$ is finite, so in the rest of the proof we will assume that $\kappa$ is infinite.
Let $(e_k)_{k \in \kappa}$ be a basis of $V$. Given $x \in V$, we denote by $(x_k)_{k \in \kappa}$
its family of coordinates in that basis, and we denote by $\supp(x)$ the support of that list, i.e.\ the (finite)
set of all $k \in \kappa$ such that $x_k \neq 0$.

Assume first that every $I$-family of norms on $V$
is dominated. Let $f : I \times \kappa \rightarrow \omega$.
For $i \in I$, set
$$N_i : x \mapsto \sup_{k \in \kappa} \,\bigl((f(i,k)+1)\,|x_k|\bigr),$$
which is obviously a norm on $V$.
Assume that some function $h : V \rightarrow \R$ dominates all norms $N_i$.
Then, there is a function $g : I \rightarrow \R$ such that
$$\forall (i,x)\in I \times V, \; N_i(x) \leq |g(i)|\,|h(x)|.$$
Applying this to every vector of $(e_k)_{k \in \kappa}$, we would find
$$\forall i \in I, \; \forall k \in \kappa, \; |f(i,k)| \leq f(i,k)+1=N_i(e_k) \leq |g(i)|\,|h(e_k)|.$$
Hence, $(i,k) \mapsto f(i,k)$ is separably dominated. We conclude that $\calP(I,\kappa)$ holds.

Conversely, assume that $\calP(I,\kappa)$ holds.
Denote by $\calP_f(\kappa)$ the set of all finite subsets of $\kappa$. Since $\kappa$ is infinite,
$\calP_f(\kappa)$ is equipotent to $\kappa$ and we deduce that $\calP(I,\calP_f(\kappa))$ holds.
Now, let $(N_i)_{i \in I}$ be an $I$-family of norms on $V$.
Choose an arbitrary norm $N'$ on $V$.
Let $J \in \calP_f(\kappa)$ and $i \in I$. The restrictions of the norms $N'$ and $N_i$ to the finite-dimensional subspace
$V_J:=\Vect(e_j)_{j \in J}$ are then equivalent, which yields a positive real number $c(i,J)$ such that
$$\forall x \in V_J, \; \forall i \in I, \; N_i(x) \leq c(i,J)\, N'(x).$$
Since $\calP(I,\calP_f(\kappa))$ holds, there are functions $f : I \rightarrow \omega$
and $g : \calP_f(\kappa) \rightarrow \omega$ such that
$$\forall (i,J) \in I \times \calP_f(\kappa), \; c(i,J) \leq f(i)\,g(J).$$
Hence,
$$\forall (i,x)\in I \times V, \; N_i(x) \leq f(i)\,g\bigl(\supp(x)\bigr) N'(x),$$
whence $(i,x) \mapsto N_i(x)$ is separately dominated. It follows from Lemma \ref{sepdomlemma} that the family of functions $(N_i)_{i \in I}$
is dominated, and we conclude by Lemma \ref{dominFunctiontoNorm} that some norm dominates all the $N_i$'s.
\end{proof}

\end{document}